\documentclass[a4paper,11pt]{amsart}
\usepackage{amsfonts}
\usepackage{amsmath}
\usepackage{amssymb}
\usepackage[a4paper]{geometry}
\usepackage{mathrsfs}
\usepackage{xcolor}
\usepackage{amsfonts,amsmath,amssymb,amsthm,stmaryrd,hyperref,bbm}
\usepackage{hyperref}
\renewcommand\eqref[1]{(\ref{#1})} %Need with hyperref
\usepackage[utf8]{inputenc}
\usepackage[T1]{fontenc}

\numberwithin{equation}{section}
\theoremstyle{plain}
\newtheorem{theorem}{Theorem}[section]

\newtheorem{lemma}[theorem]{Lemma}

\theoremstyle{definition}
\newtheorem{definition}[theorem]{Definition}
\newtheorem{remark}[theorem]{Remark}

\begin{document}
\title[Parabolic equations with Hartree-type nonlinearities]
{Fujita-type results for parabolic equations with Hartree-type nonlinearities}

\author[A. Z. Fino]{Ahmad Z. Fino}
\address{
  Ahmad Z. Fino
  \endgraf College of Engineering and Technology, \endgraf American University of the Middle East, Kuwait.}
\email{ahmad.fino@aum.edu.kw}
  
\author[B. T. Torebek ]{Berikbol T. Torebek} \address{Berikbol T. Torebek  \endgraf Institute of
Mathematics and Mathematical Modeling \endgraf 28 Shevchenko str.,
050010 Almaty, Kazakhstan.} \email{torebek@math.kz}



\begin{abstract} 
This paper investigates the critical behavior of global solutions to a parabolic equation with a Hartree-type nonlinearity of the form
$$
\left\{\begin{array}{ll}
u_{t}+(-\Delta)^{\frac{\beta}{2}} u= (\mathcal{K}\ast |u|^{p})|u|^{q},&\qquad x\in \mathbb{R}^n,\,\,\,t>0,\\\\
 u(x,0)=u_{0}(x),& \qquad x\in \mathbb{R}^n,\\
 \end{array}
 \right.
 $$
where $\beta\in(0,2]$, $n\geq1$, $p>1$, $q\geq 1$, $(-\Delta)^{\frac{\beta}{2}},\,\beta\in(0,2)$ denotes the fractional Laplacian, the symbol $\ast$ denotes the convolution operation in $\mathbb{R}^n$, and $\mathcal{K}:(0,\infty)\rightarrow(0,\infty)$ is a continuous function such that $\mathcal{K}(|\cdotp|)\in L^1_{\hbox{\tiny{loc}}}(\mathbb{R}^n)$ and is monotonically decreasing in a neighborhood of infinity. We establish conditions for the global nonexistence of solutions to the problem under consideration, thereby partially improving some results of Filippucci and Ghergu in [Discrete Contin. Dyn.
Syst. A, 42 (2022) 1817-1833] and [Nonlinear Anal., 221 (2022) 112881]. In addition, we establish local and global existence results in the case where the convolution term corresponds to the Riesz potential. Our methodology relies on the nonlinear capacity method and the fixed-point principle, combined with the Hardy–Littlewood–Sobolev inequality.
\end{abstract}

\maketitle

\section{Introduction}
This study is devoted to the analysis of global properties of the solutions for a class of semilinear parabolic equations involving nonlinear space-convolution term
\begin{equation}\label{1}
\left\{\begin{array}{ll}
u_{t}+(-\Delta)^{\frac{\beta}{2}} u= (\mathcal{K}\ast |u|^{p})|u|^{q},&\qquad x\in \mathbb{R}^n,\,\,\,t>0,\\\\
 u(x,0)=u_{0}(x),& \qquad x\in \mathbb{R}^n,\\
 \end{array}
 \right.
\end{equation}
where $\beta\in(0,2]$, $n\geq1$, $p,q>0$, $u_0\in L^1_{\hbox{\tiny{loc}}}(\mathbb{R}^n),$ and the fractional Laplacian 
$(-\Delta)^{\frac{\beta}{2}}$, with $0<\beta<2$, is defined in \eqref{fractional} below. The function $\mathcal{K}:(0,\infty)\rightarrow(0,\infty)$ is continuous, satisfies $\mathcal{K}(|\cdotp|)\in L^1_{\hbox{\tiny{loc}}}(\mathbb{R}^n)$ and there exists $R_0>1$ such that $\inf\limits_{r\in(0,R)}\mathcal{K}(r)=\mathcal{K}(R)$ for all $R>R_0$. The nonlinear convolution term $\mathcal{K}\ast|u|^p$ is the convolution between $\mathcal{K}$ and $|u|^p$ defined by
\begin{align*}(\mathcal{K}\ast |u|^p)(x)&=\int_{\mathbb{R}^n}\mathcal{K}(|x-y|)|u(y)|^p\,dy\\&=\int_{\mathbb{R}^n}\mathcal{K}(|y|)|u(x-y)|^p\,dy.\end{align*}
Typical examples of $\mathcal{K}$ are the constant functions as well as
$$\mathcal{K}(r)=r^{-\sigma},\quad\sigma\in(0,n)\qquad\hbox{or}\qquad \mathcal{K}(r)=r^{-\sigma}\ln^{\delta}(1+r),\quad\sigma\in(0,n),\,\,\delta>\sigma-n.$$ In case $$\mathcal{K}(r)=A_\sigma r^{-(n-\sigma)},\,\sigma\in(0,n),\,A_\sigma=\frac{\Gamma\left( \frac{n - \sigma}{2} \right)}{\Gamma\left( \frac{\sigma}{2} \right)\pi^{\frac n2}2^\sigma}$$ the convolution $(\mathcal{K}\ast v)(x)$ coincides with the classical Riesz potential $$A_\sigma(|x|^{-(n-\sigma)}\ast v)(x)=I_{\sigma}(v(x))=A_{\sigma}\int_{\mathbb{R}^n}\frac{v(y)}{|x-y|^{n-\sigma}}dy.$$

\subsection{Historical background}
We begin with the pioneering work of Fujita \cite{Fuj}, where the following semilinear heat equation 
\begin{equation}\label{heat}
\left\{
\begin{array}{ll}
\,\,\displaystyle{u_t-\Delta u=|u|^{p}}&\displaystyle {x\in {\mathbb{R}^n},\;t>0,}\\
\\
\displaystyle{u(x,0)=u_0(x)}&\displaystyle{x\in {\mathbb{R}^n},}
\end{array}
\right. \end{equation}
was first studied. Namely, Fujita \cite{Fuj} proved that for $p<p_F=1+2/n,$ no nonnegative global-in-time solution exists for any nontrivial initial data, whereas for $p>p_F$, global positive solutions exist provided the initial data is sufficiently small and nonnegative. The blow-up of nonnegative solutions in the critical case $p = p_F$ was established in \cite{11, 17, 22}. 
The quantity $p_F=1+2/n$ is known in the literature as the critical Fujita exponent for equation \eqref{heat}, characterizing the threshold between global existence and finite-time blow-up of solutions. Space-fractional extensions of Fujita's problem have been studied in many papers (see for example \cite{Fino5, Kirane-Guedda, KT2025, Naga, 22}).

The analysis of convolution terms within the nonlinearities of evolution equations has played a significant role in modeling diverse phenomena in both gravitational theory and quantum physics. For instance, in 1928, Hartree (see \cite{Hartee1,Hartee2,Hartee3}) considered the following equation
\begin{equation}\label{convolution1}
i\psi_{t}-\Delta \psi=(|x|^{\alpha-n}\ast \psi^2)\psi ,\qquad x\in \mathbb{R}^n,\,\,\,t>0,
\end{equation} in connection with the Schr\"{o}dinger equation in quantum physics. The stationary case of \eqref{convolution1} for $n=3$ and $\alpha=2$ is known in the literature as the Choquard equation and was introduced in \cite{Pekar} as a model in quantum theory. The stationary cases of problem \eqref{1} with $\mathcal{K}(r)=r^{\alpha-n},\,\alpha\in(0,n)$, including the critical exponents related to the existence of solutions, have been investigated in \cite{Ghergu2, Ghergu1, Moroz1, Moroz2}.

To the best of our knowledge, the first results on parabolic equations with a Hartree-type nonlinearity, defined via a convolution operation, i.e.
\begin{equation}\label{realnonlocal}
\begin{cases}
u_{t}-\Delta u\geq (\mathcal{K}\ast |u|^p)|u|^q ,\qquad {x\in \mathbb{R}^n,\,\,\,t>0,}
 \\{}\\ u(x,0)=u_{0}(x), \qquad\qquad\qquad x\in \mathbb{R}^n,
 \end{cases}
\end{equation}
were obtained only recently, in 2022, by Filippucci and Ghergu \cite{Filippucci1,Filippucci2}, where $n\geq1$, $p,q>0$, and $u_0\in L^1_{\hbox{\tiny{loc}}}(\mathbb{R}^n)$. Indeed, in \cite{Filippucci1,Filippucci2} the authors investigated even more general frameworks than \eqref{realnonlocal}, encompassing quasilinear operators such as the $m$-Laplacian, or the generalized mean curvature operator. They established that \eqref{realnonlocal} admits no nontrivial global weak solutions whenever $p+q>2$, $u_0\in L^1(\mathbb{R}^n)$,
$$\int_{\mathbb{R}^n}u_0(x)\,dx>0\qquad\hbox{and}\qquad \limsup_{R\longrightarrow\infty}\mathcal{K}(R)\,R^{\frac{2n+2}{p+q}-n}>0.$$
They exemplified these results by examining the case of pure power potentials of the form $$\mathcal{K}(r)=r^{-(n-\alpha)},\,\,\alpha\in(n-1,n),$$ and established that problem \eqref{realnonlocal} admits no global weak solutions whenever the initial data satisfy $$u_0\in L^1(\mathbb{R}^n), \int_{\mathbb{R}^n}u_0(x)\,dx>0$$ and $$2<p+q\leq\frac{2(n+1)}{2n-\alpha}=1+\frac{\alpha+2}{2n-\alpha}.$$ These results were subsequently extended in several directions: for instance, the more general quasilinear case was studied in \cite{Wang}, while extensions of these results to Riemannian manifolds and to Heisenberg groups were investigated in \cite{Jleli1} and \cite{FinoKirane}, respectively. Problems of this type for parabolic equations of the form \eqref{1} in bounded domains with Dirichlet boundary conditions have been extensively studied in \cite{Liu1, Liu2, Zhou1, Zhou2}.

It is also worth noting that Fujita-type results for semilinear parabolic equations with spatially nonlocal nonlinearity of the form $$\left(\int_{\mathbb{R}^n}K(x)u^rdx\right)^{(p-1)/r}u^q,\,\,p,q,r\geq 1$$ were studied in \cite{galak}. Here $K\geq 0$ is a measurable function on $\mathbb{R}^n.$ It has been established that if $K\in L^1(\mathbb{R}^n)$ and $p+q>1,\,r\geq 1$, then the critical exponent is given by $$p_c=2+\frac{2}{n}-q,$$ while when $q=1,$ $K\not\in L^1(\mathbb{R}^n),$ and $$C_1|x|^{-\gamma_1}\leq K(x)\leq C_2|x|^{-\gamma_2},\,\gamma_1\leq n,\,|x|\geq \rho>0,$$ the critical exponent takes the form $$p_c=1+\frac{2r}{n(r-1)+\max\{\gamma_1,0\}}.$$ Related problems in more general settings have been investigated in numerous studies (see, e.g., \cite{Jleli2, Laptev, Nabti, Souplet1, Souplet2, Zheng}).

Motivated by these results, the purpose of the present study is to investigate the global behavior of solutions to the problem \eqref{1}. In particular, our objective is not only to extend previously known results to the setting of the fractional heat equation, but also to refine and strengthen the corresponding results in the case of the classical heat equation (see Remark \ref{remarkA}).

\subsection{Main results}
Let us start with the definition of the weak solution of \eqref{1}.
\begin{definition}\textup{(Weak solution of \eqref{1})}${}$\\
Let $u_0\in L^1_{\hbox{\tiny{loc}}}(\mathbb{R}^n)$ and $T>0$. We say that $u\in L_{\hbox{\tiny{loc}}}^1((0,T)\times\mathbb{R}^n)$ is a weak solution of \eqref{1} on $[0,T)\times\mathbb{R}^n$ if
$$ (\mathcal{K}\ast |u|^p)|u|^q \in L_{\hbox{\tiny{loc}}}^1((0,T)\times\mathbb{R}^n),$$
and
\begin{equation}\label{weaksolution2}\begin{split}
\int_{\mathbb{R}^n}u(\tau,x)\psi(\tau,x)\,dx-\int_{\mathbb{R}^n}u(0,x)\psi(0,x)\,dx&=\int_0^\tau\int_{\mathbb{R}^n}(\mathcal{K}\ast|u|^p)|u|^q\psi(t,x)\,dx\,dt\\&-\int_0^\tau\int_{\mathbb{R}^n}u\,(-\Delta)^{\frac{\beta}{2}}\psi(t,x)\,dx\,dt\nonumber\\
&+\int_0^\tau\int_{\mathbb{R}^n}u\,\psi_t(t,x)\,dx\,dt,
\end{split}\end{equation}
holds for all compactly supported test function $\psi\in C^{1,2}_{t,x}([0,T)\times\mathbb{R}^n)$, and $0\leq\tau<T$. If $T=\infty$,  $u$ is called a global in time weak solution to \eqref{1}.
\end{definition}

\begin{theorem}\label{theo1}
 Let $n\geq1,$ $\beta\in(0,2],$ $p,q>0$ be such that $p+q>2$.
\begin{itemize}
\item[(i)] If $u_0\in L^1(\mathbb{R}^n)$,
$$\int_{\mathbb{R}^n}u_0(x)\,dx>0\qquad\hbox{and}\qquad \limsup_{R\rightarrow\infty}\left(\mathcal{K}(R)\,R^{-n(p+q-2)+\beta}\right)=\infty,$$
then problem \eqref{1} has no global weak solutions.
\item[(ii)] If $ u_0\in L^1_{\hbox{\tiny{loc}}}(\mathbb{R}^n)$,
$$u_0(x)\geq\varepsilon(1+|x|^2)^{-\gamma/2}\,\,\,\,\hbox{and}\,\,\,\,\,
\liminf_{R\longrightarrow\infty}\left(\mathcal{K}(R)^{-1}\,R^{\gamma(p+q-1)-n-\beta}\right)=0,$$
for some positive constant $\varepsilon>0$ and any exponent $\gamma>0$, then problem \eqref{1} has no global weak solutions.\\
\end{itemize}
\end{theorem}
\begin{remark}\label{remarkA} Noting that for the specimen choices of $\mathcal{K}$, such as $\mathcal{K}(r)=r^{-(n-\alpha)}$ with $\alpha\in(0,n)$, one can easily check that 
$$ \limsup_{R\rightarrow\infty}\left(\mathcal{K}(R)\,R^{\frac{-n(p+q-2)+\beta}{p+q}}\right)>0$$ implies $$\limsup_{R\rightarrow\infty}\left(\mathcal{K}(R)\,R^{-n(p+q-2)+\beta}\right)=\infty.$$
This verifies that our result is stronger than those obtained by Filippucci and Ghergu \cite{Filippucci1,Filippucci2}.
\end{remark}

\noindent Note that, when 
$$\mathcal{K}(r)=A_\alpha r^{-(n-\alpha)},\quad\alpha\in(0,n),\qquad A_\alpha= \frac{\Gamma\left( \frac{n - \alpha}{2} \right)}{\Gamma\left( \frac{\alpha}{2} \right)\pi^{\frac n2}2^\alpha},$$ that is
\begin{equation}\label{2}
\begin{cases}
u_{t}+(-\Delta)^{\frac{\beta}{2}} u= I_\alpha(|u|^p)|u|^{q} ,\qquad {\,x\in \mathbb{R}^n,\,\,\,t>0,}
 \\{}\\ u(x,0)=u_{0}(x), \qquad\qquad\qquad\qquad\quad x\in \mathbb{R}^n,
 \end{cases}
\end{equation}
and due to the fact that
$$\limsup_{R\rightarrow\infty}\left(\mathcal{K}(R)\,R^{-n(p+q-2)+\beta}\right)=\infty\Longleftrightarrow\lim_{R\longrightarrow\infty}R^{-n(p+q-1)+\alpha+\beta}=\infty\Longleftrightarrow-n(p+q-1)+\alpha+\beta >0,
$$
and
$$
\liminf_{R\longrightarrow\infty}\left(\mathcal{K}(R)^{-1}\,R^{\gamma(p+q-1)-n-\beta}\right)=0\Longleftrightarrow\lim_{R\longrightarrow\infty}R^{\gamma(p+q-1)-\alpha-\beta}=0\Longleftrightarrow\gamma(p+q-1)-\alpha-\beta<0,
$$
we have the following
\begin{theorem}\label{cor1}
Let $n\geq1,$ $p,q>0$, $\alpha\in(0,n)$, $\beta\in(0,2]$.\\
\begin{itemize}
\item[(i)] Assume that $\alpha+\beta>n$ and let $u_0\in L^1(\mathbb{R}^n)$ be such that
$$\int_{\mathbb{R}^n}u_0(x)\,dx>0.$$
If
$$2<p+q< 1+\frac{\beta+\alpha}{n},$$
then problem \eqref{2} has no global weak solutions.
\item[(ii)]  If 
$$u_0(x)\geq\varepsilon(1+|x|^2)^{-\gamma/2}\,\,\,\,\hbox{and}\,\,\,\,\,
2<p+q<1+\frac{\beta+\alpha}{\gamma},$$
for some positive constant $\varepsilon>0$ and any exponent $\gamma\in(0,\alpha+\beta)$, then problem \eqref{2} has no global weak solutions.
\end{itemize}
\end{theorem}
\begin{remark}\label{remark5} Theorem \ref{cor1} leads directly to the following conclusions:
\begin{description}
\item[(i)] In \cite{Filippucci1,Filippucci2}, Filippucci and Ghergu established that problem (for $\beta=2$) \eqref{2} admits no global solutions whenever $$2<p+q\leq 1+\frac{2+\alpha}{2n-\alpha},\quad\alpha\in(n-1,n),$$ but in Theorem \ref{cor1} we have the same result for $$2<p+q<1+\frac{2+\alpha}{n},\quad\alpha\in(n-2,n).$$ Since $$1+\frac{2+\alpha}{2n-\alpha}<1+\frac{2+\alpha}{n},$$ part (i) of Theorem \ref{cor1} strengthens this result by extending the admissible range of $p+q$ and $\alpha$.
\item[(ii)] It is worth noting that choosing $\gamma<n<\alpha+\beta$ allows Theorem \ref{cor1}-(ii) to extend the nonexistence result  from Theorem \ref{cor1}-(i). \end{description}
\end{remark}
Next, we give the definition of the mild solution of \eqref{2}.

\begin{definition}  A function $u\in L^\infty((0,T),X)$ is called a {\bf mild solution} of $\eqref{2}$ if $u$ has the initial data $u_0$ and satisfies the integral equation
	\begin{equation}\label{brandaq}
		u(t)=S_{\beta}(t)u_0 + \int_{0}^{t}S_{\beta}(t-\tau ) I_\alpha(|u|^{p})|u(\tau)|^{q} \,\mathrm{d}\tau.
	\end{equation}
\end{definition}
\begin{theorem}[Local existence]\label{localexistenceq}
	Let $\beta\in(0,2]$, $\alpha\in(0,n)$, $p>n/(n-\alpha)$, $q\geq1$, and $u_0 \in L^s(\mathbb{R}^{n})\cap L^\infty(\mathbb{R}^{n})$ with $n/(n-\alpha)<s<n(p-1)/\alpha$. Then there exists a time $T=T(u_0)>0$ such that problem \eqref{2} possesses a unique mild solution 
	$$u\in C([0,T],L^s(\mathbb{R}^n))\cap L^\infty((0,T),L^\infty(\mathbb{R}^{n})).$$
	 Moreover, the following properties hold:
	\begin{itemize}
		\item[$\mathrm{(i)}$] Problem \eqref{2} possesses a maximal mild solution $$u\in C([0,T_{\max}),L^s(\mathbb{R}^{n}))\cap L^\infty((0,T_{\max}),L^\infty(\mathbb{R}^{n}))$$ with $u(0)=u_0$, where $T_{\max}=T_{\max}(u_0)\leq\infty$. Furthermore, either $T_{\max}=\infty$, or else $T_{\max}<\infty$ and
	\begin{equation}\label{limblowupq}
			\lim_{t\rightarrow T_{\max}^-}(\|u(t)\|_{L^\infty(\mathbb{R}^{n})}+\|u(t)\|_{L^s(\mathbb{R}^{n})})=+\infty.
		\end{equation}
		\item[$\mathrm{(ii)}$] If $u_0\geq0$, then then the mild solution $u\geq0$ remains nonnegative for all $t\in[0,T]$.
	\end{itemize}
\end{theorem}
It is straightforward to verify that if \( u(t, x) \) is a solution of equation \eqref{2} with initial data \( u_0 \), then for all \( \lambda > 0 \), the rescaled function
\[
u_\lambda(t,x)=\lambda^{\frac{\beta+\alpha}{p+q - 1}} u(\lambda^\beta t, \lambda x)
\]
is also a solution, with initial value \( u_\lambda(0,x)=\lambda^{\frac{\beta+\alpha}{p+q - 1}} u_0(\lambda x) \). Moreover, the $L^q$-norm of the rescaled initial data satisfies
\[
\left\|u_\lambda (0) \right\|_{L^{q^{*}}} = \lambda^{\frac{\beta+\alpha}{p+q - 1} - \frac{n}{q}} \|u_0\|_{L^{q^{*}}}.
\]
This observation shows that the scaling-invariant Lebesgue exponent for equation \eqref{2}  is
 \begin{equation}\label{qscalingq}
q_{\mathrm{sc}} = \frac{n(p+q - 1)}{\beta+\alpha}. 
\end{equation}
One might therefore expect that if \( q_{\mathrm{sc}} > 1 \), that is, \( p+q > p_{\mathrm{sc}} \), where the critical exponent is given by
 \begin{equation}\label{pscalingq}
p_{\mathrm{sc}} = 1 + \frac{\beta+\alpha}{n},
\end{equation}
and if \( \|u_0\|_{L^{q_{\mathrm{sc}}}} \) is sufficiently small, then the corresponding solution should exist globally in time. However, the next result shows that this expectation does not hold.
\begin{theorem}[Global existence]\label{globalq} Let $n\geq1$, $\alpha\in(0,n)$, $0<\beta\leq 2$, $p>1$, and $q\geq1$.
If
 \begin{equation}\label{esti2q}
    p+q> 1+\frac{\beta+\alpha}{n-\alpha},
\end{equation}
and $u_0\in L^{q_{\mathrm{sc}}}(\mathbb{R}^n)\cap L^{\infty}(\mathbb{R}^n)$ (where $q_{\mathrm{sc}}$ is given by \eqref{qscalingq}), then the solution $u\in L^\infty((0,\infty),L^\infty(\mathbb{R}^n))$ of \eqref{2} exists
globally in time, provided that $\|u_0\|_{L^{q_{\mathrm{sc}}}}$ is sufficiently small.
\end{theorem}
At the same time, Theorem \ref{cor1} immediately implies the following statement regarding the results of a blow-up in finite time:
\begin{theorem}[Blow-up]\label{blowup}
Let $n\geq1,$ $\alpha\in(0,n)$, $\beta\in(0,2],$ $p>1$ and $q\geq 1$.
\begin{itemize}
\item[(i)] Assume that $\alpha+\beta>n$ and let $u_0\in L^1(\mathbb{R}^n)$ be such that
$$\int_{\mathbb{R}^n}u_0(x)\,dx>0.$$
If
$$2<p+q< 1+\frac{\beta+\alpha}{n},$$
then the mild local solutions of problem \eqref{2} blows-up in finite time.
\item[(ii)]  If 
$$u_0(x)\geq\varepsilon(1+|x|^2)^{-\gamma/2}\,\,\,\,\hbox{and}\,\,\,\,\,
2<p+q<1+\frac{\beta+\alpha}{\gamma},$$
for some positive constant $\varepsilon>0$ and any exponent $\gamma\in(0,\alpha+\beta)$, then the mild local solutions of problem \eqref{2} blows-up in finite time.
\end{itemize}
\end{theorem}
\begin{remark} 
\begin{itemize}
\item By analyzing the global existence result in Theorem \ref{globalq}, one can define $$p^*=1+\frac{\beta+\alpha}{n-\alpha}$$ as the type of critical exponent for problem \eqref{2}. However, from the blow-up result in Theorem \ref{blowup} we have a different exponent $$p_*=1+\frac{\beta+\alpha}{n}.$$ At this time we have a gap for $$p+q\in\left[p_*,p^*\right],$$ hence the question of the Fujita-type critical exponent remains open. Theorems \ref{globalq} and \ref{blowup} give an optimal result only in case $\gamma=n-\alpha$, where the critical exponent is $1+\frac{\beta+\alpha}{n-\alpha}.$
\item Without loss of generality, we may change the term $|u|^{q}$ in Theorems \ref{localexistenceq} and \ref{globalq} can be replaced by $|u|^{q-1}u$ which is a general form of Choquard type.
\end{itemize}
\end{remark}

%%%%%%%%%%%%%%%%%%%%%%%%%%%%%%%%%%%%%%%%%%%%%%%%%%%%%%%%%%%%%%%%%%%%%%%%%

\section{Preliminaries}
In this section, we present some preliminary knowledge needed in our proofs hereafter. First, we recall that the fundamental solution
$S_\beta=S_\beta(x,t)$ of the linear equation
\begin{equation}\label{linearequ+}
u_t+(-\Delta)^{\beta/2}u=0,\quad \beta\in(0,2],\;
x\in\mathbb{R}^n,\;t>0,
\end{equation}
can be written  via the Fourier transform
as follows
\begin{equation}\label{3.3}
S_\beta(x,t)=\frac{1}{(2\pi)^{n/2}}\int_{\mathbb{R}^n}e^{ix.\xi-t|\xi|^\beta}\,d\xi.
\end{equation}
It is well-known that for each $\beta\in(0,2],$ this function
satisfies
\begin{equation}\label{P_1+}
    S_\beta(1)\in L^\infty(\mathbb{R}^n)\cap
L^1(\mathbb{R}^n),\quad
S_\beta(x,t)\geq0,\quad\int_{\mathbb{R}^n}S_\beta(x,t)\,dx=1,
\end{equation}
\noindent for all $x\in\mathbb{R}^n$ and $t>0.$ Hence, using Young's inequality for the convolution and the following self-similar form  $$S_\beta(x,t)=t^{-n/\beta}S_\beta(xt^{-1/\beta},1),$$ we have
\begin{lemma}[$L^p-L^q$ estimate]\label{Lp-Lqestimate} 
	Let $\beta\in(0,2]$, and $1\leq r\leq q\leq\infty$. Then there exists a positive constant $C$ such that for every $v\in L^r(\mathbb{R}^{n})$, the following inequalities hold
	\begin{eqnarray}
		\|S_\beta(t)\ast v\|_q&\leq &Ct^{-\frac{n}{\beta}(\frac{1}{r}-\frac{1}{q})}\|v\|_r,\qquad t>0,\label{semigroup1}\\
		\|S_\beta(t)\ast v\|_r&\leq& \|v\|_r,\qquad t>0\label{semigroup2}
	\end{eqnarray}
\end{lemma}
The following Hardy-Littlewood-Sobolev is a particular case of \cite[Theorem~4.3]{LiebLoss}.
\begin{lemma}[Hardy-Littlewood-Sobolev]\label{Hardy}${}$\\
	Let $1< p< r<\infty$ and $0<\alpha<n$ with $1/p+\alpha/n=1+1/r$. Then there exists a positive constant $C=C(n,\alpha,p)>0$ such that for every $f\in L^p(\mathbb{R}^{n})$, the following inequality holds
	\begin{equation}\label{Hardy-Littlewood}
		\||x|^{-\alpha}\ast f\|_r\leq C \|f\|_p.
	\end{equation}
\end{lemma}
\begin{definition}[\cite{Kwanicki,Silvestre}]\label{def4}
\fontshape{n}
\selectfont
Let $s \in (0,1)$. Let $X$ be a suitable set of functions defined on $\mathbb{R}^n$. Then, the fractional Laplacian $(-\Delta)^s$ in $\mathbb{R}^n$ is a non-local operator given by
\begin{equation}\label{fractional}
(-\Delta)^s: \,\,v \in X  \to (-\Delta)^s v(x):= C_{n,s}\,\, p.v.\int_{\mathbb{R}^n}\frac{v(x)- v(y)}{|x-y|^{n+2s}}dy, 
\end{equation}
as long as the right-hand exists, where $p.v.$ stands for Cauchy's principal value and $$C_{n,s}:= \frac{4^s \Gamma(\frac{n}{2}+s)}{\pi^{\frac{n}{2}}\Gamma(-s)}$$ is a normalization constant.
\end{definition}

\begin{lemma}\label{lemma4}\cite[Lemma~2.4]{DaoReissig}
Let $s \in (0,1]$, and $\varphi$ be a smooth function satisfying $\partial_x^2\varphi\in L^\infty(\mathbb{R}^n)$. For any $R>0$, let $\varphi_R$ be a function defined by
$$\varphi_R(x):= \varphi(x/R) \quad \text{ for all } x \in \mathbb{R}^n.$$
Then, $(-\Delta)^s \varphi_R$ satisfies the following scaling properties:
$$(-\Delta)^s \varphi_R(x)= R^{-2s}((-\Delta)^s\varphi)(x/R) \quad \text{ for all } x \in \mathbb{R}^n. $$
\end{lemma}

%%%%%%%%%%%%%%%%%%%%%%%%%%%%%%%%%%%%%%%%%%%%%%%%%%%%%%%%%%%%

%%%%%%%%%%%%%%%%%%%%%%%%%%%%%%%%%%%%%%%%%%%%%%%%%%%%%%%%%%%%

%%%%%%%%%%%%%%%%%%%%%%%%%%%%%%%%%%%%%%%%%%%%%%%%%%%%%%%%%%%%%%%%%%%%%%

\section{Global nonexistence}

\begin{proof}[Proof of Theorem \ref{theo1}]
The proof is by contradiction. Assume that $u$ is a nonnegative global weak solution of \eqref{1}. Then $u$ satisfies the following weak formulation
\begin{eqnarray*}
&{}&\int_0^T\int_{\mathbb{R}^n}(\mathcal{K}\ast|u|^p)|u|^q\psi(t,x)\,dx\,dt+\int_{\mathbb{R}^n}u(0,x)\psi(0,x)\,dx\\
&{}&=\int_0^T\int_{\mathbb{R}^n}u\,(-\Delta)^{\frac{\beta}{2}}\psi(t,x)\,dx\,dt-\int_0^T\int_{\mathbb{R}^n}u\,\psi_t(t,x)\,dx\,dt
\end{eqnarray*}
 for all $T>0$ and all compactly supported $\psi\in C^{2,1}([0,T]\times\mathbb{R}^n)$ such that $\psi(T,\cdotp)=0$.
 Let $R$ and $T$ be large parameters in $(0,\infty)$. Let us choose 
$$\psi(t,x):= \varphi^\ell_R(x) \varphi^\ell_T(t),$$
where
$$
\varphi_R(t)= \Phi\left(\frac{|x|}{R}\right),\qquad \varphi_T(t)= \Phi\left(\frac{t}{T}\right),\qquad x\in\mathbb{R}^n,\,\, t>0,
$$
with $\ell=(2(p+q)-2)/(p+q-2)$, and $\Phi\in \mathcal{C}^\infty(\mathbb{R})$ is a smooth non-increasing function 
satisfying $\mathbbm{1}_{(-\infty,\frac{1}{2}]} \leq \Phi\leq \mathbbm{1}_{(-\infty,1]}$. Then,
\begin{eqnarray}\label{3q}
\int_0^T\int_\mathcal{B}(\mathcal{K}\ast|u|^p)|u|^q\psi(t,x)\,dx\,dt+\int_{\mathcal{B}}u_0(x)\varphi^{\ell}_R(x)\,dx&=&\int_0^T\int_{\mathcal{B}}u\,\varphi^\ell_T(t)(-\Delta)^{\frac{\beta}{2}}\varphi^{\ell}_R(x)\,dx\,dt\nonumber\\
&{}&\quad -\int_{\frac T2}^T\int_{\mathcal{B}}u\,\varphi^{\ell}_R(x)\partial_t(\varphi^{\ell}_T(t))\,dx\,dt\nonumber\\
&=:&I_1+I_2,
\end{eqnarray}
where $\mathcal{B}=\left\{x\in\mathbb{R}^n;\,\,|x|\leq R\right\}$. Let us first derive an estimate for $I_1$. Using  H\"older's inequality together with  Ju's inequality (see e.g. \cite[Appendix]{Fino5}) $$(-\Delta)^{\beta/2}\left(\varphi_{R}^\ell\right)\leq \ell\varphi_{R}^{\ell-1}(-\Delta)^{\beta/2}\varphi_{R},$$ we have
\begin{eqnarray}\label{4q}
I_1&=&\int_0^T\int_{\mathcal{B}}u\,\varphi^\ell_T(t)(-\Delta)^{\frac{\beta}{2}}\varphi^{\ell}_R(x)\,dx\,dt\nonumber\\
&\leq&\ell \int_0^T\int_{\mathcal{B}}u\,\psi^{\frac{2}{p+q}}(t,x)\psi^{-\frac{2}{p+q}}(t,x)\varphi^\ell_T(t)\varphi^{\ell-1}_R(x)\left|(-\Delta)^{\frac{\beta}{2}}\varphi_R(x)\right|\,dx\,dt\nonumber\\
&\leq&\ell\left(\int_0^T\int_{\mathcal{B}}u^{\frac{p+q}{2}}\psi dx dt\right)^{\frac{2}{p+q}}\left(\int_0^T\int_{\mathcal{B}}\varphi^{\ell}_T(t)\varphi_R(x)\,\left|(-\Delta)^{\frac{\beta}{2}}\varphi_R(x)\right|^{\frac{p+q}{p+q-2}}dxdt\right)^{\frac{p+q-2}{p+q}} .
\end{eqnarray}
Similarly, by $\partial_t\varphi^\ell_T(t)=\ell\varphi_T^{\ell-1}(t)\partial_t\varphi_T(t)$, we obtain
\begin{equation}\label{5q}
I_2\leq\ell \left(\int_{\frac T2}^T\int_{\mathcal{B}}u^{\frac{p+q}{2}}\psi\,dx\,dt\right)^{\frac{2}{p+q}}\left(\int_{\frac T2}^T\int_{\mathcal{B}}\varphi_T(t)\varphi_R^{\ell}(x)\,\left|\partial_t\varphi_T(t)\right|^{\frac{p+q}{p+q-2}}\,dx\,dt\right)^{\frac{p+q-2}{p+q}}.
\end{equation}
Inserting \eqref{4q}-\eqref{5q} into \eqref{3q}, we arrive at
\begin{eqnarray}\label{6q}
&{}&\int_0^T\int_{\mathcal{B}}(\mathcal{K}\ast|u|^p)|u|^q\psi(t,x)\,dx\,dt+\int_{\mathcal{B}}u_0(x)\varphi^{\ell}_R(x)\,dx\nonumber\\
&{}&\leq J_1\,\left(\int_0^T\int_{\mathcal{B}}u^{\frac{p+q}{2}}\psi(t,x)\,dx\,dt\right)^{\frac{2}{p+q}}+\,J_2\,\left(\int_{\frac T2}^T\int_{\mathcal{B}}u^{\frac{p+q}{2}}\psi(t,x)\,dx\,dt\right)^{\frac{2}{p+q}},
\end{eqnarray}
where
$$J_1:=\ell\left(\int_0^T\int_{\mathcal{B}}\varphi^{\ell}_T(t)\varphi _R(x)\,\left|(-\Delta)^{\frac{\beta}{2}}\varphi_R(x)\right|^{\frac{p+q}{p+q-2}}\,dx\,dt\right)^{\frac{p+q-2}{p+q}},$$
and
$$J_2:=\ell\left(\int_{\frac T2}^T\int_{\mathcal{B}}\varphi _T(t)\varphi^{\ell}_R(x)\,\left|\partial_t\varphi_T(t)\right|^{\frac{p+q}{p+q-2}}\,dx\,dt\right)^{\frac{p+q-2}{p+q}} .$$
Let us estimate $J_2$. We have
\begin{eqnarray}\label{7q}
J_2&\leq& \ell\left(\int_{\mathcal{B}}\varphi^{\ell}_R(x)\,dx\right)^{\frac{p+q-2}{p+q}}\left(\int_0^T\Phi\left(\frac{t}{T}\right)\,\left|\partial_t\Phi\left(\frac{t}{T}\right)\right|^{\frac{p+q}{p+q-2}}\,dt\right)^{\frac{p+q-2}{p+q}}\notag\\
&\leq&C\,R^{\frac{n(p+q-2)}{p+q}}T^{-\frac{2}{p+q}}\left(\int_0^1\Phi(\tilde{t})\,\left|\Phi^{\prime}(\tilde{t})\right|^{\frac{p+q}{p+q-2}}\,d\widetilde{t}\right)^{\frac{p+q-2}{p+q}}\notag\\
&\leq& C\,R^{\frac{n(p+q-2)}{p+q}}T^{-\frac{2}{p+q}},
\end{eqnarray}
where we have used the change of variables
$$\widetilde{x}=\frac{x}{R},\qquad \widetilde{t}=\frac{t}{T}.$$
Similarly, by using Lemma \ref{lemma4}, we have
\begin{equation}\label{8q}\begin{split}
J_1&= \ell\left(\int_0^T\varphi^\ell_T(t)\,dt\right)^{\frac{p+q-2}{p+q}}\left(\int_{\mathcal{B}}\varphi_R(x)\,\left|(-\Delta)^{\frac \beta 2}\varphi_R(x)\right|^{\frac{p+q}{p+q-2}}\,dx\right)^{\frac{p+q-2}{p+q}}\\&\leq C\,T^{\frac{p+q-2}{p+q}}R^{\frac{n(p+q-2)}{p+q}-\beta}.\end{split}
\end{equation}
By \eqref{7q}-\eqref{8q}, we get from \eqref{6q} that
\begin{equation}\label{9q}\begin{split}
\int_0^T\int_{\mathcal{B}}(\mathcal{K}\ast|u|^p)|u|^q\psi(t,x)\,dx\,dt+\int_{\mathcal{B}}u_0(x)\varphi^{\ell}_R(x)\,dx& \leq C\,T^{\frac{p+q-2}{p+q}}R^{\frac{n(p+q-2)}{p+q}-\beta}\textbf{I}^{\frac{2}{p+q}}\\&+C\, R^{\frac{n(p+q-2)}{p+q}}T^{-\frac{2}{p+q}}\textbf{J}^{\frac{2}{p+q}},\end{split}
\end{equation}
where
\[
\textbf{I}:= \int_0^T\int_{\mathcal{B}}u^{\frac{p+q}{2}}\psi(t,x)\,dx\,dt\qquad\hbox{and}\qquad\textbf{J}:= \int_{\frac T2}^T\int_{\mathcal{B}}u^{\frac{p+q}{2}}\psi(t,x)\,dx\,dt.
\]
To estimate the left-hand side of \eqref{9q}, we have
\begin{align*}(\mathcal{K}\ast|u|^p)(x)&=\int_{\mathbb{R}^n}\mathcal{K}(|x-y|)|u(y)|^p\,dy\\&\geq \int_{\mathcal{B}}\mathcal{K}(|x-y|)|u(y)|^p\,dy.\end{align*}
Note that $|x|\leq R$ and $|y| \leq R$ on $\mathcal{B}$, then
$$|x-y|\leq 2 R,\quad\hbox{for all}\,\,x,\,y\in\mathcal{B},$$
which implies that
$$\mathcal{K}(|x-y|)\geq \mathcal{K}(2R),\quad\hbox{on}\,\,\mathcal{B},$$
for all $R\gg1$, namely $2R>R_0$. So,
$$(\mathcal{K}\ast|u|^p)(x)\geq  \mathcal{K}(2R)\int_{\mathcal{B}}|u(y)|^p\,dy,\quad\hbox{for all}\,\,x\in\mathcal{B}.$$
Then
\begin{eqnarray}\label{11q}
\int_{\mathcal{B}}(\mathcal{K}\ast|u|^p)|u(x)|^q\psi(t,x)\,dx&\geq&  \mathcal{K}(2R)\int_{\mathcal{B}}\int_{\mathcal{B}}|u(y)|^p|u(x)|^q\psi(t,x)\,dy\,dx\nonumber\\
&\geq&\mathcal{K}(2R)\int_{\mathbb{R}^n}\int_{\mathbb{R}^n}|u(y)|^p\psi(t,y)|u(x)|^q\psi(t,x)\,dy\,dx,
\end{eqnarray}
where we have used that $\psi\leq 1$ and $\psi(\cdotp,t)\equiv 0$ outside of $\mathcal{B}$ for all $t\geq0$.
On the other hand, using Cauchy-Schwarz' inequality, we have
\begin{eqnarray*}
&{}&\int_{\mathbb{R}^n}\int_{\mathbb{R}^n}|u(y)|^{\frac{p+q}{2}}\psi(t,y)|u(x)|^{\frac{p+q}{2}}\psi(t,x)\,dy\,dx\nonumber\\
&{}&=\int_{\mathbb{R}^n}\int_{\mathbb{R}^n}|u(y)|^{\frac{p}{2}}\psi^{\frac 12}(t,y)|u(x)|^{\frac{q}{2}}\psi^{\frac 12}(t,x)|u(y)|^{\frac{q}{2}}\psi^{\frac 12}(t,y)|u(x)|^{\frac{p}{2}}\psi^{\frac 12}(t,x)\,dy\,dx\nonumber\\
&{}&\leq\left(\int_{\mathbb{R}^n}\int_{\mathbb{R}^n}|u(y)|^p\psi(t,y)|u(x)|^q\psi(t,x)\,dy\,dx\right)^{\frac 12}\left(\int_{\mathbb{R}^n}\int_{\mathbb{R}^n}|u(y)|^q\psi(t,y)|u(x)|^p\psi(t,x)\,dy\,dx\right)^{\frac 12}\nonumber\\
&{}&=\int_{\mathbb{R}^n}\int_{\mathbb{R}^n}|u(y)|^p\psi(t,y)|u(x)|^q\psi(t,x)\,dy\,dx,
\end{eqnarray*} that is
\begin{equation}\label{12q}\begin{split}
&\int_{\mathbb{R}^n}\int_{\mathbb{R}^n}|u(y)|^{\frac{p+q}{2}}\psi(t,y)|u(x)|^{\frac{p+q}{2}}\psi(t,x)\,dy\,dx\\\leq & \int_{\mathbb{R}^n}\int_{\mathbb{R}^n}|u(y)|^p\psi(t,y)|u(x)|^q\psi(t,x)\,dy\,dx.\end{split}    
\end{equation}
Combining \eqref{11q}, and \eqref{12q}, and letting
\begin{align*}K(t):&=\int_{\mathcal{B}}(u(t,x))^{\frac{p+q}{2}}\psi(t,x)\,dx\\&=\int_{\mathbb{R}^n}(u(t,x))^{\frac{p+q}{2}}\psi(t,x)\,dx,\quad\hbox{for almost all}\,\,t\geq0,\end{align*}
we infer that
\begin{eqnarray*}
\int_{\mathcal{B}}(\mathcal{K}\ast|u|^p)|u(x)|^q\psi(t,x)\,dx&\geq& \mathcal{K}(2R)\int_{\mathbb{R}^n}\int_{\mathbb{R}^n}|u(y)|^{\frac{p+q}{2}}\psi(t,y)|u(x)|^{\frac{p+q}{2}}\psi(t,x)\,dy\,dx\nonumber\\
&{}&= \mathcal{K}(2R)\left(\int_{\mathbb{R}^n}|u(x)|^{\frac{p+q}{2}}\psi(t,x)\,dx\right)^2\nonumber\\
&{}&=\mathcal{K}(2R)\,K^2(t).
\end{eqnarray*}
Hence
\begin{equation}\label{120q}
\int_0^T\int_{\mathcal{B}}(\mathcal{K}\ast |u|^p)|u|^q \psi(t,x)\,dx\,dt \geq \mathcal{K}(2R)\int_0^TK^2(t)\,dt.
\end{equation}
On the other hand, by the Cauchy-Schwarz inequality, we have
\begin{equation}\label{130q}\begin{split}
 \textbf{I}^{\frac{2}{p+q}}&=\left(\int_0^T\int_{\mathcal{B}}(u(t,x))^{\frac{p+q}{2}}\psi(t,x)\,dx\,dt\right)^{\frac{2}{p+q}}\\&\leq T^{\frac{1}{p+q}}\left(\int_0^T K^2(t)\,dt\right)^{\frac{1}{p+q}}.\end{split}
\end{equation}
Similarly, 
\begin{equation*}
\textbf{J}^{\frac{2}{p+q}}\leq T^{\frac{1}{p+q}}\left(\int_{\frac T2}^T K^2(t)\,dt\right)^{\frac{1}{p+q}}.
\end{equation*}
By \eqref{9q}, \eqref{120q}, and \eqref{130q}, we get
\begin{equation}\label{10q}\begin{split}\mathcal{K}(2R)\int_0^TK^2(t)\,dt+2\int_{\mathcal{B}}u_0(x)\varphi^{\ell}_R(x)\,dx
 &\leq C\,T^{\frac{p+q-1}{p+q}}R^{\frac{n(p+q-2)}{p+q}-\beta}\left(\int_{0}^T K^2(t)\,dt\right)^{\frac{1}{p+q}}\\&+C\, R^{\frac{n(p+q-2)}{p+q}}T^{-\frac{1}{p+q}}\left(\int_{\frac T2}^T K^2(t)\,dt\right)^{\frac{1}{p+q}}.\end{split}
\end{equation}
By choosing $R=T^{\frac 1\beta}$, it follows from \eqref{10q} that
$$
 \mathcal{K}(2T^{\frac 1\beta})\int_0^TK^2(t)\,dt+\int_{\mathcal{B}}u_0(x)\varphi^{\ell}_R(x)\,dx\leq C\,T^{\frac{n(p+q-2)-\beta}{\beta(p+q)}}\left(\int_{0}^T K^2(t)\,dt\right)^{\frac{1}{p+q}},
$$
that is,
$$
\int_0^TK^2(t)\,dt+\frac{1}{ \mathcal{K}(2T^{\frac 1\beta})}\int_{\mathcal{B}}u_0(x)\varphi^{\ell}_R(x)\,dx\leq C\frac{T^{\frac{n(p+q-2)-\beta}{\beta(p+q)}}}{ \mathcal{K}(2T^{\frac 1\beta})}\left(\int_{0}^T K^2(t)\,dt\right)^{\frac{1}{p+q}}.
$$
By Young's inequality $ab\leq \frac{1}{2}a^{p+q}+Cb^{\frac{p+q}{p+q-1}}$, we arrive at
$$
\frac{1}{ \mathcal{K}(2T^{\frac 1\beta})}\int_{\mathcal{B}}u_0(x)\varphi^{\ell}_R(x)\,dx\leq C\frac{T^{\frac{n(p+q-2)-\beta}{\beta(p+q-1)}}}{ (\mathcal{K}(2T^{\frac 1\beta}))^{\frac{p+q}{p+q-1}}},
$$
which implies that
\begin{equation}\label{14q}
\int_{\mathcal{B}}u_0(x)\varphi^{\ell}_R(x)\,dx\leq \frac{C}{ \left(\mathcal{K}(2T^{\frac 1\beta})T^{\frac{-n(p+q-2)+\beta}{\beta }}\right)^{\frac{1}{p+q-1}}}.
\end{equation}
As
$$\limsup_{R\rightarrow\infty}\left(\mathcal{K}(R)\,R^{-n(p+q-2)+\beta}\right)=\infty,$$
then
\begin{align*}\liminf_{R\rightarrow\infty}\int_{\mathcal{B}}u_0(x)\varphi^{\ell}_R(x)\,dx&\lesssim  \left(\liminf_{R\rightarrow\infty}\left(\mathcal{K}(R)^{-1}R^{n(p+q-2)-\beta}\right)\right)^{\frac{1}{p+q-1}}\\&=0.\end{align*}
Using $u_0\in L^1(\mathbb{R}^n)$ and the Lebesgue convergence theorem, we arrive at
$$0<\int_{\mathbb{R}^n}u_0(x)\,dx\leq 0;$$
contradiction.\\

\noindent (ii) In this case, as $R>1$, we have the following estimate
\begin{align*} \int_{\mathcal{B}}u_0(x)\varphi^\ell_{R}(x)\,dx&\geq \int_{\overline{\mathcal{B}}}u_0(x)\,dx\\&\geq\varepsilon \int_{\overline{\mathcal{B}}}(1+|x|^2)^{-\gamma/2}\,dx\\&\geq \varepsilon\,C \int_{\overline{\mathcal{B}}}\left(R^{2}+ R^{2}\right)^{-\gamma/2}\,dx\\&=\varepsilon\,C R^{n-\gamma},\end{align*}
where $\overline{\mathcal{B}}:=\left\{x\in\mathbb{R}^n;\,\, |x|\leq {R}/{2}\right\}$. Therefore, by repeating the same calculation as in case (i) (see \eqref{14q}) and recalling that $T^{\frac 1\beta}=R$ , we get
$$
\varepsilon\,C R^{n-\gamma} \leq C\left(\mathcal{K}(2R)^{-1}R^{n(p+q-2)-\beta}\right)^{\frac{1}{p+q-1}},
$$
that is,
$$
\varepsilon\lesssim \left(\mathcal{K}(2R)^{-1}R^{\gamma(p+q-1)-n-\beta}\right)^{\frac{1}{p+q-1}},
$$
and so,
$$
0<\varepsilon\lesssim \left(\liminf_{R\rightarrow\infty}\left(\mathcal{K}(R)^{-1}R^{\gamma(p+q-1)-n-\beta}\right)\right)^{\frac{1}{p+q-1}}=0;
$$
contradiction. The proof is complete.
\end{proof}
%%%%%%%%%%%%%%%%%%%%%%%%%%%%%%%%%%%%%%%%%%%%%%%%%%%%%%%%%%%%%%%%%%%%%%

%%%%%%%%%%%%%%%

\section{Local existence.}

\begin{proof}[Proof of Theorem \ref{localexistenceq}] The proof is divided into several steps.\\
	{\it Step 1. Fixed-point argument.} Let $T>0$ be fixed. We define the Banach space
$$E_T=L^\infty((0,T),L^s(\mathbb{R}^n)\cap L^\infty(\mathbb{R}^{n})).$$
The norm on $E_T$ is defined by
$$\|u\|_{E_T}=\sup_{t \in (0,T)}\|u(t)\|_{L^{s}\cap L^\infty}= \sup_{t \in (0,T)}\left(\|u(t)\|_{L^{s}}+\|u(t)\|_{L^{\infty}}\right).$$
We choose $R\geq  \|u_0\|_{L^s\cap L^\infty}$. In order to use the Banach fixed-point theorem, we introduce the following nonempty complete metric space
	\begin{equation}\label{norm}
		B_T(R)=\{u\in E_T:\,\,\|u\|_{E_T}\leq 2 R\},
	\end{equation}
	equipped with the distance $d(u,v)=\|u-v\|_E $.	For $u\in B_{T}(R)$, we define $\Lambda(u)$ by
	\begin{equation}\label{solop}
		\Lambda(u)(t):=S_{\beta}(t)u_0 + \int_{0}^{t}S_{\beta}(t-\tau )  I_\alpha(|u|^{p})|u|^{q-1}u(\tau ) \,\mathrm{d}\tau .
	\end{equation}
	Let us prove that  $\Lambda: B_{T}(R) \rightarrow B_{T}(R)$. Using \eqref{semigroup2}, we obtain for any $u \in B_{T}(R)$, 
	\begin{align*}
		\|\Lambda(u)(t)\|_{L^{s}} & \leq \|S_{\beta}(t)u_0\|_{L^{s}} + \left\| \int_{0}^{t}S_{\beta}(t-\tau)  I_\alpha(|u|^p)|u|^q(\tau)\,\mathrm{d}\tau \right\|_{L^{s}}  \\
				&\leq \|u_0\|_{L^s}+  \int_{0}^{t} \|I_\alpha (|u|^p)|u|^q(\tau)\|_{L^s} \, \mathrm{d}\tau\\
				&\leq \|u_0\|_{L^s}+\int_{0}^{t}  \|u(\tau)\|^q_{L^\infty} \|I_\alpha (|u|^p)(\tau)\|_{L^s} \, \mathrm{d}\tau,
				\end{align*}	
		for all $t\in (0,T)$. Applying Lemma \ref{Hardy} with the identity $1/q_1=\alpha/n+1/s$, and noting that the condition $n/(n-\alpha)<s$ guarantees $q_1>1$, we get
		\begin{align*}
		\|\Lambda(u)(t)\|_{L^{s}} &\leq \|u_0\|_{L^s}+   C\int_{0}^{t}\|u(\tau)\|^q_{L^\infty} \||u(\tau)|^p\|_{L^{q_1}} \, \mathrm{d}\tau.
				\end{align*}	
			Since $s<n(p-1)/\alpha$ implies $1/q_1<p/s$, then $p>s/q_1$ and therefore
		\begin{align}\label{fix1}
		\|\Lambda(u)(t)\|_{L^{s}} &\leq \|u_0\|_{L^s}+  C \int_{0}^{t} \|u(\tau)\|^q_{L^\infty}\|u(\tau)\|^{p-\frac{s}{q_1}}_{L^{\infty}} \||u(\tau)|^{\frac{s}{q_1}}\|_{L^{q_1}} \, \mathrm{d}\tau\notag\\
		&= \|u_0\|_{L^s}+   C\int_{0}^{t} \|u(\tau)\|^{p+q-\frac{s}{q_1}}_{L^{\infty}} \|u(\tau)\|^{\frac{s}{q_1}}_{L^{s}} \, \mathrm{d}\tau\notag\\
		&\leq \|u_0\|_{L^s} + C \|u\|^{p+q}_{E_T}  T\notag\\
		&\leq R+ C 2^{p+q}R^{p+q}  T\notag\\
		&\leq 2R,
				\end{align}
	for all $t\in (0,T)$, for a sufficiently small $T>0$ (depending on $R$). Moreover, by \eqref{semigroup1}-\eqref{semigroup2}, we have
	\begin{align*}
		\|\Lambda(u)(t)\|_{L^{\infty}} & \leq \|S_{\beta}(t)u_0\|_{L^{\infty}} + \left\| \int_{0}^{t}S_{\beta}(t-\tau)  I_\alpha(|u|^p)|u|^q(\tau) \,\mathrm{d}\tau \right\|_{L^{\infty}}  \\
				&\leq \|u_0\|_{L^\infty}+C   \int_{0}^{t} (t-\tau)^{-\frac{n}{\beta r}} \|I_\alpha (|u|^p)|u|^q(\tau)\|_{L^{r}} \, \mathrm{d}\tau\\
				&\leq \|u_0\|_{L^\infty}+C   \int_{0}^{t} \|u(\tau)\|^q_{L^\infty} (t-\tau)^{-\frac{n}{\beta r}} \|I_\alpha (|u|^p)(\tau)\|_{L^{r}} \, \mathrm{d}\tau,
				\end{align*}	
		for all $t\in (0,T)$.	Using Lemma \ref{Hardy} with $1/\tilde{q}=\alpha/n+1/r$, we get
		\begin{align*}
		\|\Lambda(u)(t)\|_{L^{\infty}} &\leq \|u_0\|_{L^\infty}+   C\int_{0}^{t} (t-\tau)^{-\frac{n}{\beta r}}\|u(\tau)\|^q_{L^\infty}\||u(\tau)|^p\|_{L^{\tilde{q}}} \, \mathrm{d}\tau.
				\end{align*}	
			By choosing $1/\tilde{q}<\min\{p/s,\,(\beta+\alpha)/n\}$, we infer that $n/(\beta r)<1$ and $p>s/\tilde{q}$. So,
		\begin{align}\label{fix2}
		\|\Lambda(u)(t)\|_{L^{\infty}} &\leq \|u_0\|_{L^\infty}+  C \int_{0}^{t}(t-\tau)^{-\frac{n}{\beta r}} \|u(\tau)\|^q_{L^\infty}\|u(\tau)\|^{p-\frac{s}{\tilde{q}}}_{L^{\infty}} \||u(\tau)|^{\frac{s}{\tilde{q}}}\|_{L^{\tilde{q}}} \, \mathrm{d}\tau\notag\\
		&= \|u_0\|_{L^\infty}+   C\int_{0}^{t}(t-\tau)^{-\frac{n}{\beta r}} \|u(\tau)\|^{p+q-\frac{s}{\tilde{q}}}_{L^{\infty}} \|u(\tau)\|^{\frac{s}{\tilde{q}}}_{L^{s}} \, \mathrm{d}\tau\notag\\
		&\leq \|u_0\|_{L^\infty} + C \|u\|^{p+q}_{E_T}  T^{1-\frac{n}{\beta r}}\notag\\
		&\leq R+ C 2^{p+q}R^{p+q}  T^{1-\frac{n}{\beta r}}\notag\\
		&\leq 2R,
				\end{align}
	for all $t\in (0,T)$, for a sufficiently small $T>0$ (depending on $R$). By combining \eqref{fix1}-\eqref{fix2}, we conclude that $\|\Lambda(u)\|_{E_T}\leq 2 R$, that is  $\Lambda(u) \in B_{T}(R)$.\\
	Similarly, one can see that $\Lambda$ is a contraction. For $u,v \in B_T(R)$, we have
\begin{eqnarray*}
		\|\Lambda(u)(t)-\Lambda(v)(t)\|_{L^{s}}&\leq& \int_{0}^{t} \|I_\alpha (|u|^{p})|u|^{q-1}u(\tau)-I_\alpha(|v|^{p})|v|^{q-1}v(\tau)\|_{L^s} \, \mathrm{d}\tau\\
		&\leq& \int_{0}^{t} \|I_\alpha (||u(\tau)|^{p}-|v(\tau)|^{p}|)|u|^q\|_{L^s} \, \mathrm{d}\tau\\
		&{}&\,+ \int_{0}^{t} \|I_\alpha (|v|^p)||u(\tau)|^{q-1}u(\tau)-|v(\tau)|^{q-1}v(\tau)|\|_{L^s} \, \mathrm{d}\tau\\
		&\leq& \int_{0}^{t} \|u(\tau)\|^q_{L^\infty}\|I_\alpha (||u(\tau)|^{p}-|v(\tau)|^{p}|)\|_{L^s} \, \mathrm{d}\tau\\
		&{}&\,+ \int_{0}^{t}\||u(\tau)|^{q-1}u(\tau)-|v(\tau)|^{q-1}v(\tau)\|_{L^\infty}  \|I_\alpha (|v|^p)\|_{L^s} \, \mathrm{d}\tau
		\end{eqnarray*}
for all $t\in (0,T)$.  Using Lemma \ref{Hardy} with $1/q_1=\alpha/n+1/s$, we get
		\begin{align*}
		\|\Lambda(u)(t)-\Lambda(v)(t)\|_{L^{s}} &\leq   C\int_{0}^{t} \|u\|^q_{L^\infty}\||u(\tau)|^{p}-|v(\tau)|^{p}\|_{L^{q_1}} \, \mathrm{d}\tau\\
		&\quad+C \int_{0}^{t}\||u(\tau)|^{q-1}u(\tau)-|v(\tau)|^{q-1}v(\tau)\|_{L^\infty}  \||v|^p\|_{L^{q_1}} \, \mathrm{d}\tau.				\end{align*}	
		As 
		$$||u(\tau)|^{p}-|v(\tau)|^{p}|\leq C(|u(\tau)|^{p-1}+|v(\tau)|^{p-1})|u(\tau)-v(\tau)|,$$
		$$||u(\tau)|^{q-1}u(\tau)-|v(\tau)|^{q-1}v(\tau)|\leq C(|u(\tau)|^{q-1}+|v(\tau)|^{q-1})|u(\tau)-v(\tau)|$$
		 and 
		 $$\frac{1}{q_1}=\frac{\alpha}{n}+\frac{1}{s},$$
			by applying H\"older's inequality, we obtain	
			$$\||u(\tau)|^{p}-|v(\tau)|^{p}\|_{L^{q_1}}\leq C(\||u(\tau)|^{p-1}\|_{L^{\frac{n}{\alpha}}}+\||v(\tau)|^{p-1}\|_{L^{\frac{n}{\alpha}}})\|u(\tau)-v(\tau)\|_{L^{s}},$$
			and
			\begin{align*}\||u(\tau)|^{q-1}u(\tau)-|v(\tau)|^{q-1}v(\tau)\|_{L^\infty}&\leq (\|u(\tau)\|^{q-1}_{L^\infty}+\|v(\tau)\|^{q-1}_{L^\infty})\|u(\tau)-v(\tau)\|_{L^{\infty}}\\&\leq 2^q R^{q-1}\|u-v\|_{E_T}.\end{align*}	
				Since $s<n(p-1)/\alpha$ implies $p-1>s\alpha/n$, we conclude that
				$$\||u(\tau)|^{p-1}\|_{L^{\frac{n}{\alpha}}}\leq \|u(\tau)\|^{p-1-\frac{s\alpha}{n}}_{L^{\infty}}\|u(\tau)\|^{\frac{s\alpha}{n}}_{L^s}\leq (2R)^{p-1},$$
			and
				$$\||v(\tau)|^{p-1}\|_{L^{\frac{n}{\alpha}}}\leq \|v(\tau)\|^{p-1-\frac{s\alpha}{n}}_{L^{\infty}}\|v(\tau)\|^{\frac{s\alpha}{n}}_{L^s}\leq (2R)^{p-1}.$$
			Therefore,
$$
			\||u(\tau)|^{p}-|v(\tau)|^{p}\|_{L^{q_1}}\leq  C 2^p\,R^{p-1}\|u(\tau)-v(\tau)\|_{L^{s}}.
		$$
		Moreover, 
		$$ \||v|^p\|_{L^{q_1}} \leq \|v(\tau)\|^{p-\frac{s}{q_1}}_{L^{\infty}} \|v(\tau)\|^{\frac{s}{q_1}}_{L^{s}} \leq (2R)^p.$$
			It follows that
						\begin{equation}\label{fix3}
		\|\Lambda(u)(t)-\Lambda(v)(t)\|_{L^{s}} \leq C 2^{p+q}R^{p+q-1} T\|u-v\|_{E_T}\leq \frac{1}{2}d(u,v),
				\end{equation}
	for all $t\in (0,T)$, for a sufficiently small $T>0$ (depending on $R$). In addition, 		\begin{align*}
		\|\Lambda(u)(t)-\Lambda(v)(t)\|_{L^{\infty}} & =\left\| \int_{0}^{t}S_{\beta}(t-\tau) \left(I_\alpha (|u|^{p})|u|^{q-1}u(\tau)-I_\alpha(|v|^{p})|v|^{q-1}v(\tau)\right)\,\mathrm{d}\tau \right\|_{L^{\infty}}  \\
		&\leq \int_{0}^{t}(t-\tau)^{-\frac{n}{\beta r}}  \|I_\alpha (||u(\tau)|^{p}-|v(\tau)|^{p}|)|u|^q\|_{L^r} \, \mathrm{d}\tau\\
		&\quad+ \int_{0}^{t} (t-\tau)^{-\frac{n}{\beta r}} \|I_\alpha (|v|^p)||u(\tau)|^{q-1}u(\tau)-|v(\tau)|^{q-1}v(\tau)|\|_{L^r} \, \mathrm{d}\tau\\
				&\leq C   \int_{0}^{t} (t-\tau)^{-\frac{n}{\beta r}} \|u(\tau)\|^q_{L^\infty} \|I_\alpha (||u|^{p}-|v|^{p}|)(\tau)\|_{L^{r}} \, \mathrm{d}\tau\\
				&\quad+ \int_{0}^{t} (t-\tau)^{-\frac{n}{\beta r}} \|I_\alpha (|v|^p)\|_{L^r} \||u(\tau)|^{q-1}u(\tau)-|v(\tau)|^{q-1}v(\tau)|\|_{L^\infty} \, \mathrm{d}\tau,
				\end{align*}	
		for all $t\in (0,T)$.	Using Lemma \ref{Hardy} with $1/\tilde{q}=\alpha/n+1/r$, we get
		\begin{align*}
		\|\Lambda(u)(t)-\Lambda(v)(t)\|_{L^{\infty}} &\leq  C\int_{0}^{t} (t-\tau)^{-\frac{n}{\beta r}}\|u(\tau)\|^q_{L^\infty} \||u(\tau)|^{p}-|v(\tau)|^{p}\|_{L^{\tilde{q}}} \, \mathrm{d}\tau\\
		&\quad+ \int_{0}^{t} (t-\tau)^{-\frac{n}{\beta r}} \||v|^p\|_{L^{\tilde{q}}} \||u(\tau)|^{q-1}u(\tau)-|v(\tau)|^{q-1}v(\tau)|\|_{L^\infty} \, \mathrm{d}\tau.
				\end{align*}	
			As 
		$$|u(\tau)|^{p}-|v(\tau)|^{p}\leq C(|u(\tau)|^{p-1}+|v(\tau)|^{p-1})|u(\tau)-v(\tau)|$$
		 and 
		 $$\frac{1}{\tilde{q}}=\frac{\alpha}{n}+\frac{1}{r},$$
			by applying H\"older's inequality, we obtain	
			$$\||u(\tau)|^{p-1}-|v(\tau)|^{p-1}\|_{L^{\tilde{q}}}\leq C(\||u(\tau)|^{p-1}\|_{L^{\frac{n}{\alpha}}}+\||v(\tau)|^{p-1}\|_{L^{\frac{n}{\alpha}}})\|u(\tau)-v(\tau)\|_{L^{r}}.$$	
				Since $s<n(p-1)/\alpha$ implies $p-1>s\alpha/n$, we conclude that
				$$\||u(\tau)|^{p-1}\|_{L^{\frac{n}{\alpha}}}\leq \|u(\tau)\|^{p-1-\frac{s\alpha}{n}}_{L^{\infty}}\|u(\tau)\|^{\frac{s\alpha}{n}}_{L^s}\leq (2R)^{p-1},$$
			and
				$$\||v(\tau)|^{p-1}\|_{L^{\frac{n}{\alpha}}}\leq \|v(\tau)\|^{p-1-\frac{s\alpha}{n}}_{L^{\infty}}\|v(\tau)\|^{\frac{s\alpha}{n}}_{L^s}\leq (2R)^{p-1}.$$
				Furthermore, by choosing $$1/\tilde{q}<\min\{p\alpha/n(p-1),\, (\beta+\alpha)/n\},$$ we infer that $n/(\beta r)<1$ and $r>n(p-1)/\alpha>s,$ so
				\begin{align*}\|u(\tau)-v(\tau)\|_{L^{r}}&\leq \|u(\tau)-v(\tau)\|^{1-\frac{s}{r}}_{L^\infty}\|u(\tau)-v(\tau)\|^{\frac{s}{r}}_{L^s}\\&\leq \|u-v\|_{E_T}.\end{align*}
			Noting that $$1/\tilde{q}<\min\{p\alpha/n(p-1),\, (\beta+\alpha)/n\}<\min\{p/s,\,(\beta+\alpha)/n\},$$ due to $s<n(p-1)/\alpha$. Therefore,
$$
			\||u(\tau)|^{p}-|v(\tau)|^{p}\|_{L^{\tilde{q}}}\leq C2^p\,R^{p-1}\|u-v\|_{E_T}.
		$$
		Moreover, 
			\begin{align*}\||u(\tau)|^{q-1}u(\tau)-|v(\tau)|^{q-1}v(\tau)\|_{L^\infty}&\leq (\|u(\tau)\|^{q-1}_{L^\infty}+\|v(\tau)\|^{q-1}_{L^\infty})\|u(\tau)-v(\tau)\|_{L^{\infty}}\\&\leq 2^q R^{q-1}\|u-v\|_{E_T},\end{align*}
			and
			$$\||v|^p\|_{L^{\tilde{q}}}\leq \|v(\tau)\|^{p-\frac{s}{\tilde{q}}}_{L^{\infty}} \||v(\tau)|^{\frac{s}{\tilde{q}}}\|_{L^{\tilde{q}}}\leq (2R)^p.$$	
		It follows that
		\begin{equation}\label{fix4}\begin{split}
		\|\Lambda(u)(t)-\Lambda(v)(t)\|_{L^{\infty}}& \leq C 2^{p+q}R^{p+q-1} T^{1-\frac{n}{\beta r}}\|u-v\|_{E_T}\\&\leq \frac{1}{2}d(u,v),\end{split}
				\end{equation}
	for all $t\in (0,T)$, for a sufficiently small $T>0$ (depending on $R$). By combining \eqref{fix3}-\eqref{fix4}, we conclude that 
$$
		d(\Lambda(u),\Lambda(v)) \leq \frac{1}{2}d(u,v).
$$
	Consequently, by the contraction principle, $\Lambda$  has a unique fixed point $u$ in $B_T(R)$.\\
	
	\noindent {\it Step 2. Uniqueness in $E_T$.} We are going  to extend this uniqueness to $E_T$. Let $u,v \in E_T$.  Then, for sufficiently large $\tilde{R}>0$ and small $\tilde{T}>0$, we have $u,v \in B_{\tilde{T}}(\tilde{R})$. As a result, $u(t)= v(t)$ for small $t>0$, which, by a standard continuation argument, extends to the entire interval $(0,T)$. For more details, see e.g. \cite{FinoViana}.\\
	
	\noindent {\it Step 3. Regularity.} One can easily check that $f\in L^1((0,T),L^s(\mathbb{R}^{n}))$, with
	$$f(t):=I_\alpha(|u|^{p})|u|^{q-1}u(t),\qquad \hbox{for all}\,\,t\in (0,T).$$
	Applying \cite[Lemma~4.1.5]{CH}, using the continuity of the semigroup $S_{\beta}(t)$,  we conclude that $u\in C([0,T],L^s(\mathbb{R}^{n}))$.\\
	
	\noindent {\it Step 4. Maximal solution and blow-up alternative.} Using the uniqueness of the mild solution, we conclude the existence
	of a solution on a maximal interval $[0,T_{\max}),$ where
	$$
	T_{\max}:=\sup\left\{T>0\;;\;\text{$\exists$ a mild solution $u\in C([0,T],L^s(\mathbb{R}^{n}))\cap L^\infty((0,T),L^\infty(\mathbb{R}^{n}))$
		to \eqref{2}}\right\}.$$
	Obviously, $T_{\max}\leq \infty$. To prove that $\|u(t)\|_{L^s(\mathbb{R}^{n})\cap L^\infty(\mathbb{R}^{n})}\rightarrow\infty$ as $t\rightarrow T_{\max},$
	whenever $T_{\max}<\infty,$ we proceed by contradiction. If
	$$\liminf_{t\rightarrow T_{\max}}\|u(t)\|_{L^s(\mathbb{R}^{n})\cap L^\infty(\mathbb{R}^{n})}=:L<\infty,$$
	then there exists a time sequence $\{t_m\}_{m\geq0}$ tending to $T_{\max}$ as $m\rightarrow\infty$ and such that
	$$\sup_{m\in\mathbb{N}}\|u(t_m)\|_{L^s\cap L^\infty}\leq L+1.$$
	Using again a fixed-point argument with $u(t_m)$ as initial condition, one can deduce that there exists $T(L + 1) > 0$, depends on $L+1$, such that the solution
	$u(t)$ can be extended on the interval $[t_m, t_m + T(L + 1)]$ for any $m\geq0$. Thus, by the definition of the maximality time, $T_{\max}\geq t_m+T(L+1)$, for any $m\geq0$. We get the desired contradiction by letting $m\rightarrow\infty$.\\
	
	\noindent {\it Step 5. Positivity of solutions.} Assume $u_0\geq0$.  In this case, we can construct a nonnegative solution on some interval $[0,T]$ by applying the fixed point argument within the positive cone $E_T^+=\{u\in E_T;\;u\geq0\}$, and using the positivity of $S_\beta(t)$ \eqref{P_1+}. By the uniqueness of solutions, it then follows that $u(t)\geq0$ for all $t\in(0,T_{\max}).$\\

\end{proof}

%%%%%%%%%%%%%%%%%%%%%%%%%%%%%%%%%%%%%%%%%%%%%%%%%%%%%%%%%%%%%%%%%%%%%%%%%%%%%%%
\section{Global existence}

\begin{proof}[Proof of Theorem \ref{globalq}] Since $$p+q>1+(\beta+\alpha)/(n-\alpha),$$ it is possible to choose a positive constant $q^*>0$ such that
\begin{equation}\label{estiA}
    \frac{\beta+\alpha}{\beta(p+q-1)}-\frac{1}{p+q}<\frac{n}{\beta
    q^*}< \frac{\beta+\alpha}{\beta(p+q-1)}\quad \text{with} \quad q^*>p+q.
\end{equation}
From this, it follows that 
\begin{equation}\label{estiB}
    q^*>\frac{n(p+q-1)}{\beta+\alpha}=q_{\mathrm{sc}}>1.
\end{equation}
We then define
\begin{equation}\label{estiC}\begin{split}
    \beta^*:&=\frac{n}{\beta q_{\mathrm{sc}}}-\frac{n}{\beta
    q^*}\\&=\frac{\beta+\alpha}{\beta(p+q-1)}-\frac{n}{\beta
    q^*}.\end{split}
\end{equation}
Therefore, based on relations (\ref{estiA})-(\ref{estiC}), the following relations hold
\begin{equation}\label{estiD}
    \beta^*>0,\qquad 1-\frac{n(p+q-1)}{\beta q^*}+\frac{\alpha}{\beta} -(p+q-1)\beta^*=0,\qquad\hbox{and}\qquad (p+q)\beta^*<1.
\end{equation}
 Since $u_0\in L^{q_{\mathrm{sc}}}(\mathbb{R}^{n})$, applying Lemma \ref{Lp-Lqestimate} with $q>q_{\mathrm{sc}}$, and using \eqref{estiC}, we
get
\begin{equation}\label{estiE}
    \sup_{t>0}t^{\beta^*}\|S_{\beta}(t)u_0\|_{L^{q^*}}\leq
   C \|u_0\|_{L^{q_{\mathrm{sc}}}}=:\rho<\infty.
\end{equation}
Set
\begin{equation}\label{estiF}
    \mathbb{X}:=\left\{u\in
    L^\infty((0,\infty),L^{q^*}(\mathbb{R}^{n}));\;\sup_{t>0}t^{\beta^*}\|u(t)\|_{L^{q^*}}\leq\delta\right\},
\end{equation}
where $\delta>0$ is chosen to be sufficiently small. For $u,v\in \mathbb{X}$, we define the metric
\begin{equation}\label{estiG}
    d_{\mathbb{X}}(u,v):=\sup_{t>0}t^{\beta^*}\|u(t)-v(t)\|_{L^{q^*}}.
\end{equation}
It is straightforward to verify that $(\mathbb{X},d)$ is a nonempty complete metric space. For $u\in \mathbb{X}$, we define the mapping $\Phi(u)$ by
\begin{equation}\label{estiH}
    \Phi(u)(t):=S_{\beta}(t)u_0 + \int_{0}^{t}S_{\beta}(t-\tau) I_\alpha(|u|^{p})|u|^{q-1}u(\tau) \,\mathrm{d}\tau,\quad \text{for all}\,\, t\geq0.
\end{equation}
Let us now verify that the operator $\Phi:  \mathbb{X} \rightarrow \mathbb{X}$. By employing inequalities \eqref{estiE} and \eqref{estiF}, along with Lemma \ref{Lp-Lqestimate}, we derive the following estimate for any $u \in  \mathbb{X}$, 
$$
    t^{\beta^*}\|\Phi(u)(t)\|_{L^{q^*}}\leq \rho+\,Ct^{\beta^*}\int_0^t(t-\tau)^{-\frac{n}{\beta}\left(\frac{1}{r_1}-\frac{1}{q^*}\right)}\|I_\alpha(|u|^p)(\tau)|u|^q(\tau)\|_{L^{r_1}}\,d\tau,
   $$
   for any $1<r_1<q^*$. Using the assumption $q^*>q$, one may apply H\"older's inequality with $\frac{1}{q_1}+\frac{q}{q^*}=\frac{1}{r_1}$, then
   $$
    t^{\beta^*}\|\Phi(u)(t)\|_{L^{q^*}}\leq \rho+\,Ct^{\beta^*}\int_0^t(t-\tau)^{-\frac{n}{\beta}\left(\frac{1}{r_1}-\frac{1}{q^*}\right)}\|I_\alpha(|u|^p)(\tau)\|_{L^{q_1}}\|u(\tau)\|^q_{L^{q^*}}\,d\tau.
   $$
    By employing the assumption $q^*>p$, and using Lemma \ref{Hardy} with $p/q^*=\alpha/n+1/q_1$, we get
\begin{eqnarray}\label{estiI}
    t^{\beta^*}\|\Phi(u)(t)\|_{L^q}&\leq& \rho+\,Ct^{\beta^*}\int_0^t(t-\tau)^{-\frac{n}{\beta}\left(\frac{p+q-1}{q^*}-\frac{\alpha}{n}\right)}\|u(\tau)\|^{p+q}_{L^{q^*}}\,d\tau\nonumber\\
    &\leq&\rho+\,C\delta^{p+q} t^{\beta^*}\int_0^t(t-\tau)^{-\frac{n}{\beta}\left(\frac{p+q-1}{q^*}-\frac{\alpha}{n}\right)}\tau^{-\beta^*(p+q)}\,d\tau.
\end{eqnarray}
Now, using the parameter constraints in \eqref{estiA} and \eqref{estiD}, together with the condition $(p+q)\beta^*<1$, the integral becomes
\begin{equation}\label{estiJ}
\int_0^t(t-\tau)^{-\frac{n}{\beta}\left(\frac{p+q-1}{q^*}-\frac{\alpha}{n}\right)}\tau^{-\beta^*(p+q)}\,d\tau=C t^{-\beta^*},
\end{equation}
valid for all $t\geq0$. It then follows from estimates \eqref{estiI} and \eqref{estiJ} that
\begin{equation}\label{estiK}
    t^{\beta^*}\|\Phi(u)(t)\|_{L^{q^*}}\leq \rho+C\delta^{p+q}.
\end{equation}
Hence, if $\rho$ and $\delta$ are chosen sufficiently small such that $\rho+C\delta^{p+q}\leq\delta$, it follows that $\Phi(u)\in\mathbb{X}$, i.e., $\Phi: \mathbb{X}\rightarrow \mathbb{X}$. A similar argument shows that, under the same smallness assumptions on $\rho$ and $\delta$, the operator $\Phi$ is a strict contraction. Therefore, it has a unique fixed point $u\in \mathbb{X}$ which corresponds to a mild solution of equation \eqref{2}.

We now aim to prove that $u\in L^\infty((0,\infty),L^\infty(\mathbb{R}^{n}))$. We begin by showing that $u\in L^\infty((0,T),L^\infty(\mathbb{R}^{n}))$ for some sufficiently small $T>0$. Indeed, the previous argument ensures uniqueness in the space $ \mathbb{X}_T,$ where for any $T>0,$
$$
 \mathbb{X}_T:=\left\{u\in
    L^\infty((0,T),L^{q^*}(\mathbb{R}^{n}));\;\sup_{0<t<T}t^{\beta^*}\|u(t)\|_{L^{q^*}}\leq\delta\right\}.
$$
Let $\tilde{u}$ denote the local solution of \eqref{2} established in Theorem \ref{localexistenceq}. From inequality \eqref{estiB}, we know that $q_{\mathrm{sc}}<q<\infty$, which implies 
$$u_0\in L^\infty(\mathbb{R}^{n})\cap L^{q_{\mathrm{sc}}}(\mathbb{R}^{n})\subset L^\infty(\mathbb{R}^{n})\cap L^{q^*}(\mathbb{R}^{n}).$$ 
Moreover, using $$p+q>1+(\beta+\alpha)/(n-\alpha),$$ we have $$n/(n-\alpha)<q_{\mathrm{sc}}<q^*<n(p+q-1)/\alpha.$$ Thus, Theorem \ref{localexistenceq}  guarantees that $\tilde{u}\in L^\infty((0,T_{\max}),L^\infty(\mathbb{R}^{n})\cap L^{q^*}(\mathbb{R}^{n}))$. Consequently, for sufficiently small $T>0$, we have
$$\sup\limits_{t\in(0,T)}t^{\beta^*}\|\tilde{u}(t)\|_{L^{q^*}}\leq\delta.$$
Due to the uniqueness of solutions in $\mathbb{X}_T$, it follows that $u=\tilde{u}$ on
$[0,T]$, leading to the conclusion that $$u\in L^\infty((0,T),L^\infty(\mathbb{R}^{n})\cap L^{q^*}(\mathbb{R}^{n})).$$

To extend the regularity to $[T,\infty)$, we employ a bootstrap argument. For $t>T,$ we express $u(t)$ as
\begin{eqnarray*}
  u(t)-S_{\beta}(t)u_0 &=&
  \int_0^TS_{\beta}(t-\tau)I_\alpha(|u|^{p})|u|^{q-1}u(\tau)\,d\tau+\int_T^tS_{\beta}(t-\tau)I_\alpha(|u|^{p})|u|^{q-1}u(\tau)\,d\tau\\
   &\equiv& I_1(t)+I_2(t).
\end{eqnarray*}
Since $u\in L^\infty((0,T),L^\infty(\mathbb{R}^{n})\cap L^{q^*}(\mathbb{R}^{n}))$, and $$n/(n-\alpha)<q^*<n(p+q-1)/\alpha,$$ we immediately have $I_1\in
L^\infty((T,\infty),L^\infty(\mathbb{R}^{n})).$ Furthermore, by similar estimates used in the fixed-point argument and noting that $$t^{-\beta^*}\leq
T^{-\beta^*}<\infty,$$ we also get $I_1\in L^\infty((T,\infty),L^{q^*}(\mathbb{R}^{n}))$.
Next, from \eqref{estiB}, we observe that $q^*>q_{\mathrm{sc}}$, which guarantees the existence of some $r\in(q^*,\infty]$ such that
\begin{equation}\label{estiL}
\frac{n}{\beta}\left(\frac{p+q}{q^*}-\frac{\alpha}{n}-\frac{1}{r}\right)<1.
\end{equation}
For $T<t$, since $u\in
L^\infty((0,\infty),L^{q^*}(\mathbb{R}^{n})),$ we obtain
$$u^{p+q}\in L^\infty((T,t),L^{\frac{q^*}{p+q}}(\mathbb{R}^{n})).$$
 Thus, applying Lemma \ref{Lp-Lqestimate}, condition \eqref{estiL}, and Lemma \ref{Hardy} with $p/q^*=\alpha/n+1/q_1$, we deduce that $I_2\in
L^\infty((T,\infty),L^r(\mathbb{R}^{n}))$.
Since the terms $S_{\beta}(t)u_0$ and $I_1$ belong to
$$L^\infty((T,\infty),L^\infty(\mathbb{R}^{n}))\cap
L^\infty((T,\infty),L^{q^*}(\mathbb{R}^{n}))\subseteq L^\infty((T,\infty),L^r(\mathbb{R}^{n})),$$ we conclude that $u\in
L^\infty((T,\infty),L^r(\mathbb{R}^{n}))$. Iterating this procedure a finite
number of times, eventually leads to $u\in
L^\infty((T,\infty),L^\infty(\mathbb{R}^{n}))$. This completes the proof. 
\end{proof}

%%%%%%%%%%%%%%%%%%%%%%%%%%%%%%%%%%%%%%%%%%%%%%%%%%%%%%%%%%%%

%\backmatter
%\bmsection*{Author contributions}

%This is an author contribution text. This is an author contribution text. This is an author contribution text. This is an author contribution text. This is an author contribution text.

\section*{Acknowledgments}
Ahmad Fino is supported by the Research Group Unit, College of Engineering and Technology, American University of the Middle East. Berikbol T. Torebek is supported by the Science Committee of the Ministry of Education and Science of the Republic of Kazakhstan (Grant No. AP23483960).

\section*{Declaration of competing interest} The authors declare that there is no conflict of interest.

\section*{Data Availability Statements} The manuscript has no associated data.

\end{document}